\documentclass[final,1p,times]{elsarticle-ppn}
\usepackage{amssymb}
\usepackage{amsmath}
\usepackage{graphicx}

\journal{}

\newtheorem{theorem}{Theorem}
\newtheorem{lemma}[theorem]{Lemma}
\newtheorem{proposition}[theorem]{Proposition}
\newtheorem{corollary}[theorem]{Corollary}

\newtheorem{remark}{\it Remark}
\newenvironment{proof}{\par\smallskip\noindent\emph{Proof.\/} }{\unskip\null\hfill$\square$\par\medskip}
\newcommand{\R}{{\mathbb R}}
\newcommand{\be}[1]{\begin{equation}\label{#1}}
\newcommand{\ee}{\end{equation}}
\newcommand{\nrmrd}[2]{\left\|#1\right\|_{L^{#2}(\R^2)}}
\newcommand{\nrmn}[2]{\left\|#1\right\|_{L^{#2}(n_\infty\,dx)}}

\begin{document}\begin{frontmatter}
\title{Asymptotic behaviour for small mass in the two-dimensional parabolic-elliptic Keller-Segel model}

\author[B]{Adrien Blanchet}
\ead{adrien.blanchet@univ-tlse1.fr}
\author[D]{Jean Dolbeault}
\ead{dolbeaul@ceremade.dauphine.fr}
\author[E]{Miguel Escobedo}
\ead{miguel.escobedo@ehu.es}
\author[F]{Javier Fern{\'a}ndez}
\ead{fcojavier.fernandez@unavarra.es}

\address[B]{GREMAQ (UMR CNRS no. 5604 et INRA no. 1291), Universit\'e de Toulouse 1, Manufacture des Tabacs,\\ Aile J.J. Laffont, 21 all\'ee de Brienne, 31000 Toulouse, France.}

\address[D]{CEREMADE (UMR CNRS no. 7534), Universit\'e Paris-Dauphine, Place de Lattre de Tassigny,\\ 75775 Paris C\'edex~16, France.}

\address[E]{Departamento de Matem\'aticas, Facultad de Ciencias y Tecnolog\'\i a, Universidad del Pa\'{\i}s Vasco, Barrio Sarriena s/n, 48940 Lejona (Vizcaya), Spain.}

\address[F]{Departamento Autom\'atica y Computaci\'on, Universidad P\'ublica de Navarra, Campus Arrosad\'{\i}a s/n,\\ 31.006 Pamplona, Spain}

\begin{abstract}
The Keller-Segel system describes the collective motion of cells that are attracted by a chemical substance and are able to emit it. In its simplest form, it is a conservative drift-diffusion equation for the cell density coupled to an elliptic equation for the chemo-attractant concentration. This paper deals with the rate of convergence towards a unique stationary state in self-similar variables, which describes the intermediate asymptotics of the solutions in the original variables. Although it is known that solutions globally exist for any mass less $8\pi\,$, a smaller mass condition is needed in our approach for proving an exponential rate of convergence in self-similar~variables.
\end{abstract}

\end{frontmatter}

\section{Introduction and main results}\label{Sec:Intro}

In its simpler form, the Keller and Segel system reads
\be{Eqn:Keller-Segel}\left\{\begin{array}{ll}
\displaystyle \frac{\partial u}{\partial t}=\Delta u-\nabla \cdot (u\,\nabla v)\qquad & x\in\R^2\,,\;t>0\;,\vspace{.3cm}\\
\displaystyle -\Delta v=u\qquad & x\in\R^2\,,\;t>0\;,\vspace{.3cm}\\
u(\cdot,t=0)=n_0\geq 0&x\in\R^2\,.
\end{array}\right.\ee
Throughout this paper, we shall assume that
\be{eq:2}
n_0\in L^1_+(\R^2,(1+|x|^2)\,dx)\;,\quad n_0\log n_0\in L^1(\R^2,dx)\;,\quad\mbox{and}\quad M:=\int_{\R^2}n_0(x)\,dx<8\,\pi\;.
\end{equation}
These conditions are sufficient to ensure that a solution in a distribution sense exists globally in time and satisfies $M=\int_{\R^2} u(x,t)\,dx$ for any $t\ge 0\,$, see \cite{JL,MR2103197,BDP}. In dimension $d=2\,$, the Green kernel associated to the Poisson equation is a logarithm and we shall consider only the solution given by $v=-\frac 1{2\pi}\,\log|\cdot|*u\,$. Such a non-linearity is critical in the sense that the system is globally invariant under scalings. To study the asymptotic behaviour of the solutions, it is therefore more convenient to work in self-similar variables.
Define the rescaled functions $n$ and~$c$~by
\be{Eqn:ChangeVar}
u(x,t)=\frac 1{R^2(t)}\,n\left(\frac x{R(t)},\tau(t)\right) \quad\mbox{and}\quad v(x,t)=c\left(\frac x{R(t)},\tau(t)\right)
\ee
with $R(t)=\sqrt{1+2t}$ and $\tau(t)=\log R(t)\,$. The rescaled system is
\be{Eqn:Keller-Segel-Rescaled}\left\{\begin{array}{ll}
\displaystyle \frac{\partial n}{\partial t}=\Delta n-\nabla \cdot (n\,(\nabla c-x))\qquad & x\in\R^2\,,\;t>0\;,\\
\displaystyle c=-\frac 1{2\pi}\,\log |\cdot|*n\qquad & x\in\R^2\,,\;t>0\;,\vspace{.3cm} \\
n(\cdot,t=0)=n_0\geq 0&x\in\R^2\,.
\end{array}\right.\ee
Under Assumptions (\ref{eq:2}), it has been proved in \cite{BDP} that
\[
\lim_{t\to\infty}\nrmrd{n(\cdot,\cdot+t)-n_\infty}1=0\quad\mbox{and}\quad \lim_{t\to\infty}\nrmrd{\nabla c(\cdot,\cdot+t)-\nabla c_\infty}2=0
\]
where $(n_\infty,c_\infty)$ is the unique solution of
\[
n_\infty=M\,\frac{e^{\,c_\infty-|x|^2/2}}{\int_{\R^2}e^{c_\infty-|x|^2/2}\,dx}= -\Delta c_\infty\;,\quad\mbox{with}\quad c_\infty=-\frac 1{2\pi}\log |\cdot|*n_\infty\;.
\]
Moreover, $n_\infty$ is smooth and radially symmetric. The uniqueness has been established in \cite{MR2249579}. As $|x| \to +\infty$, $n_\infty$ is dominated by $e^{-(1-\epsilon)|x|^2/2}$ for any $\epsilon \in (0,1)$, see~\cite[Lemma~4.5]{BDP}. From the bifurcation diagram of $\nrmrd{n_\infty}\infty$ as a function of $M$, it follows that
\be{Eqn:BifurcationDiagram}\lim_{M\to 0_+}\nrmrd{n_\infty}\infty=0\;.
\ee

Under the assumption that the mass of the initial data is small enough, we first obtain estimates of the time decay rate of the $L^p$-norms of the solution $u$ of (\ref{Eqn:Keller-Segel}). Similar bounds have been obtained in several papers on Keller-Segel models such as \cite{MR2035446,MR2333473,MR2404241} (also see references therein). The interested reader may refer to \cite{biler2008pel,Rac08} for recent results relating the parabolic-parabolic and the parabolic-elliptic Keller-Segel systems. Nevertheless none of these previous works deals with~\eqref{Eqn:Keller-Segel}. See Remark \ref{estimationlp} below for more details. In a second step we prove the convergence of $n(t)$ to $n_\infty$ in the weighted Sobolev space $H^1(e^{|x|^2/4}dx)$ as $t\to +\infty\,$. Finally, we establish our main result, an exponential rate of convergence of $n(t)$ to $n_\infty$ in $L^2(n_\infty^{-1})$:
\begin{theorem}\label{Thm:Main}
There exists a positive constant $M^*$ such that, for any initial data $n_0\in L^2(n_\infty^{-1}\,dx)$ of mass $M<M^*$ satisfying \eqref{eq:2}, the rescaled Keller-Segel system \eqref{Eqn:Keller-Segel-Rescaled} has a unique solution $n\in C^0(\R^+,L^1(\R^2))\cap L^\infty((\tau,\infty)\times\R^2)$ for any $\tau>0\,$. Moreover, there are two positive constants, $C$ and $\delta\,$, such that
\[
\int_{\R^2}|n(t, x)-n_\infty(x)|^2\, \frac{dx}{n_\infty(x)} \le C\,e^{-\,\delta\, t}\quad\forall\;t>0\;.
\]
As a function of $M\,$, $\delta$ is such that $\lim_{M\to 0_+}\delta(M)=1\,$.
\end{theorem}

\begin{remark}\rm
As it has been proved in \cite{MR2103197,BDP,MR2436186}, the condition $M\le 8\, \pi$ is necessary and sufficient for the global existence of the solutions of \eqref{Eqn:Keller-Segel} under Assumption \eqref{eq:2}. The extra smallness condition in Theorem \ref{Thm:Main} appears at two levels in our proof:
\begin{itemize}
\item[1.] We first prove a uniform decay estimate of the solution of \eqref {Eqn:Keller-Segel} by the \emph{method of the trap.\/} Our estimates and the version of the Hardy-Littlewood-Sobolev (HLS) inequality we use require that $M<M_1$ for some positive, explicit constant $M_1\,$. This question is dealt with in Section~\ref{Sec:Decay}.
\item[2.] Rates of convergence in self-similar variables are given by the \emph{spectral gap\/} of a linearised operator, denoted by $\mathcal L\,$, which is associated to \eqref{Eqn:Keller-Segel-Rescaled}. This gap is estimated by a perturbation method, which gives two further restrictions on $M\,$. See Sections~\ref{Sec:f} and \ref{Sec:ExponentialConvergence}.
\end{itemize}
The first occurrence of an extra smallness condition, in the proof of the sharp time decay of the $L^p$ norms, is not surprising. It appears in several similar estimates as for example in \cite{MR2035446,MR2333473,MR2404241} and references therein. On the other hand, the estimate of the spectral gap of the linearised operator~$\mathcal L$ is rather crude. See Remark~\ref{Rem:galmathcalL} for more comments in this direction.

Under a smallness condition for the mass, we shall also obtain a uniqueness result for the solutions of \eqref{Eqn:Keller-Segel-Rescaled}, see Section~\ref{Sec:ExponentialConvergence}. For sake of simplicity, we shall speak of \emph{the} solution of \eqref{Eqn:Keller-Segel-Rescaled}, but, in the preliminary results, \emph{the} solution has to be understood as \emph{a} solution of the system which is achieved as a limit of an approximation procedure, as in \cite{JL,BDP}. 

Our results are actually stronger than the ones stated in Theorem~\ref {Thm:Main}. We can indeed consider any solution of \eqref{Eqn:Keller-Segel-Rescaled} as in \cite{BDP}:
\begin{eqnarray*}
&&n\in C^0(\R^+,L^1(\R^2))\;,\\
&&n\,\log n\,,\; n\,|x|^2\in L^\infty(\R^+,L^1(\R^2))\;,\\
&&2\nabla\sqrt n+x\,\sqrt n-\sqrt n\,\nabla c\in L^1(\R^+,L^2(\R^2))\;,
\end{eqnarray*}
and prove all \emph{a priori} estimates by standard but tedious truncation methods that we shall omit in this paper.
\end{remark}

\section{Decay Estimates of $u(t)$ in $L^\infty(\R^2)$}\label{Sec:Decay}

In this section we consider the Keller-Segel system \eqref{Eqn:Keller-Segel}, in the original variables.
\begin{lemma}\label{Lem:Uniform-u}
There exists a positive constant $M_1$ such that, for any mass $M<M_1\,$, there is a positive constant $C=C(M)$ such that, if $u\in C^0(\R^+,L^1(\R^2))\cap L^\infty(\R^+_{\rm loc}\times\R^2)$ is a solution of \eqref{Eqn:Keller-Segel} with initial datum $n_0$ satisfying \eqref{eq:2}, then
\[
\nrmrd{u(t)}\infty \le C\, t^{-1}\quad\forall\;t>0\;.
\]
\end{lemma}
\begin{proof} The result of Lemma~\ref{Lem:Uniform-u} is based on the \emph{method of the trap,\/} which amounts to prove that $H(t\nrmrd{u(\cdot,t)}\infty,M)\le 0$ where $z\mapsto H(z,M)$ is a continuous function which is negative on $[0,z_1)$ and positive on $(z_1,z_2)$ for some $z_1$, $z_2$ such that $0<z_1<z_2<\infty\,$. Since $t\mapsto t\nrmrd{u(\cdot,t)}\infty$ is continuous and takes value $0$ at $t=0$, this means that $t\nrmrd{u(\cdot,t)}\infty\le z_1\le z_0(M)$ for any $t\ge 0$, where $H(z_0(M),M)=\sup_{z\in[z_1,z_2]}H(z,M)\ge 0\,$. See Fig.~\ref{fig:1}.

Fix some $t_0 > 0\,$. By Duhamel's formula, a solution of \eqref{Eqn:Keller-Segel} can be written as
\be{Eqn:Duhamelov}
u(x,t_0+t)=\int_{\R^2} N(x-y,t)\,u(y,t_0)\;dy + \int_0^t\!\int_{\R^2} N(x-y,t-s)\;\nabla \cdot\left[u(y,t_0+s)\,\nabla v(y,t_0+s)\right]\;dy\;ds
\ee
where $N(x,t)=\frac 1{4\pi t}\,e^{-|x|^2/(4t)}$ denotes the heat kernel. Next observe that
\[
\int_0^t\!\int_{\R^2}\kern -3pt N(x-y,t-s)\;\nabla\cdot\left[u(y,t_0+s)\,\nabla v(y,t_0+s)\right]\;dy\;ds=\kern -3pt\sum_{i=1,2}\int_0^t\frac{\partial N}{\partial x_i}(\cdot,t-s)*\Big[\Big(u\,\frac{\partial v}{\partial x_i}\Big) (\cdot, t_0+s)\Big]\;ds\;.
\]
Taking $L^\infty$ norms in (\ref{Eqn:Duhamelov}) with respect to the space variable, we arrive at
\[
\nrmrd{u(\cdot ,t_0+t)}\infty \le \frac 1{4\pi t}\nrmrd{u(\cdot,t_0)}1 + \sum_{i=1,2}\int_0^t \nrmrd{\frac {\partial N}{\partial x_i} (\cdot,t-s)*\Big[\Big(u\,\frac{\partial v}{\partial x_i}\Big) (\cdot, t_0+s)\Big]\,}\infty\;ds\;.
\]
We now consider the convolution term. By Young's inequality and because of the expression for the kernel $N$, we can bound it using $\kappa_\sigma=\nrmrd{\partial N/\partial x_i\,(\cdot, 1)}\sigma$ by
\begin{multline*}
\int_0^t \nrmrd{\frac {\partial N}{\partial x_i} (\cdot,t-s)*\Big[\Big(u\,\frac{\partial v}{\partial x_i}\Big) (\cdot, t_0+s)\Big]\,}\infty\;ds\\
\le \int_0^t \left\| \frac {\partial N}{\partial x_i}(\cdot,t-s)\right\|_{L^\sigma(\R^2)}\, \left\| \Big(u\,\frac{\partial v}{\partial x_i}\Big)(\cdot, t_0+s) \right\|_{L^\rho(\R^2)}\;ds
\\
=\kappa_\sigma\int_0^t (t-s)^{-(1-\frac 1{\sigma}) -\frac 12} \left\| \Big(u\,\frac{\partial v}{\partial x_i}\Big)(\cdot, t_0+s) \right\|_{L^\rho(\R^2)}\;ds
\end{multline*}
where $1/\sigma + 1/\rho =1\,$. To enforce integrability later, we impose $\sigma < 2\,$. On the one hand
\[
\left\| \Big(u\,\frac{\partial v}{\partial x_i}\Big)(\cdot , t_0+s) \right\|_{L^\rho(\R^2)} \le \left\| u (\cdot, t_0 + s) \right\|_{L^p(\R^2)}\,\left\| \frac{\partial v}{\partial x_i} (\cdot, t_0+s) \right\|_{L^q(\R^2)}
\]
with $1/p + 1/q = 1/\rho\,$, by H\"older's inequality, whereas, on the other hand, 
\[
\left\| \frac{\partial v}{\partial x_i} (\cdot, t_0+s) \right\|_{L^q(\R^2)} \le \frac{C_{\rm HLS}}{2\pi}\, \left\| u (\cdot, t_0+s)\right\|_{L^r(\R^2)}
\]
with $1/r - 1/q = 1/2\,$, by the HLS inequality. Here $\nabla v$ is given by the convolution of $u$ with the function $x\mapsto -x_i/(2\pi|x|^2)$ and $C_{\rm HLS}$ denotes the optimal constant for the HLS inequality. Collecting all these estimates and using the fact that $\nrmrd{u(\cdot, t)}1 = M$ for any $t \ge 0$, we arrive at
\begin{eqnarray*}
&&\hspace*{-30pt}\nrmrd{u(\cdot ,t_0+t)}\infty-\frac M{4\pi t}
\\
&\le& \frac{\kappa_\sigma\,C_{\rm HLS}}\pi\, M^{\frac 1{p} +\frac 1{r}} \int_0^t (t-s)^ {-(1-\frac 1{\sigma}) -\frac 12} \nrmrd{u (\cdot,t_0 + s)}\infty^{2-\frac 1p-\frac 1r}ds
\\
&&\hspace*{6pt}= \frac{\kappa_\sigma\,C_{\rm HLS}}\pi\, M^{\frac 1{p} +\frac 1{r}} \int_0^t (t-s)^{\frac 1\sigma-\frac 32}\,(t_0+s)^{\frac 1p+\frac 1r-2}\, \Big[ (t_0+s) \nrmrd{u (\cdot, t_0 + s)}\infty \Big]^{2-\frac 1p-\frac 1r}ds\;.
\end{eqnarray*}
Now take $t_0 = t\,$, and multiply the inequality by $2t\,$ to get
\begin{multline*}
2t\,\nrmrd{u(\cdot ,2t)}\infty-\frac M{2\pi}\\
\le\frac{2\,\kappa_\sigma\,C_{\rm HLS}}\pi\, M^{\frac 1{p} +\frac 1{r}}\,t\!\int_0^t (t-s)^{\frac 1\sigma-\frac 32}\,(t+s)^{\frac 1p+\frac 1r-2}\, \Big[ (t+s) \nrmrd{u (\cdot, t + s)}\infty \Big]^{2-\frac 1p-\frac 1r}ds\;.
\end{multline*}
Observe that for any $t>0$ we have
\[
\sup_{0\le s \le t} (t+s) \nrmrd{u (\cdot, t + s)}\infty \le \sup_{0\le s \le t} 2s\,\nrmrd{u (\cdot, 2s)}\infty=:\psi(t)\;,
\]
whereas $\frac 1\sigma-\frac 32=-\frac 1p-\frac 1r$ and 
\[
t\int_0^t (t-s)^{\frac 1\sigma-\frac 32}\,(t+s)^{\frac 1p+\frac 1r-2}ds\;=\;\frac\sigma{2-\sigma}\;.
\]
{} From Duhamel's formula \eqref{Eqn:Duhamelov}, it follows that $u\in C^0(\R^+,L^\infty(\R^2)$ and $\psi$ is continuous. Hence we have
\[
\psi (t) \le \frac M{2\pi} + C_0\,\big(\psi (t) \big)^\theta\quad\mbox{with}\quad C_0=\frac{2\,\kappa_\sigma\,C_{\rm HLS}}\pi\, M^{\frac 1{p} +\frac 1{r}}\,\frac{\sigma}{2-\sigma}\;,\quad\theta=2-\frac 1p-\frac 1r\;.
\]
Consider the function $H(z,M) = z-C_0\,z^\theta - M/(2\pi)\,$, so that $H(\psi(t),M)\le 0\,$ and notice that $\theta>1\,$. For $M>0$ fixed, $z \mapsto H(z,M)$ achieves its maximum $H(z_0(M),M)=\tfrac{\theta-1}\theta\,(C_0\,\theta)^{1/(1-\theta)}-\tfrac M{2\pi}$ at $z = z_0 (M) = (C_0\,\theta)^{1/(1-\theta)}$. For $M$ small enough, as we shall see below, $H(z_0(M),M) > 0\,$. Since $\psi$ is continuous and $\psi(0)=0$ then $\psi(t)<z_0(M)$ for any $t\ge 0\,$. This provides an $L^\infty$ estimate on $\psi$ which is uniform in $t\ge 0\,$. 
\begin{figure}[ht]\begin{center}\includegraphics[height=4cm]{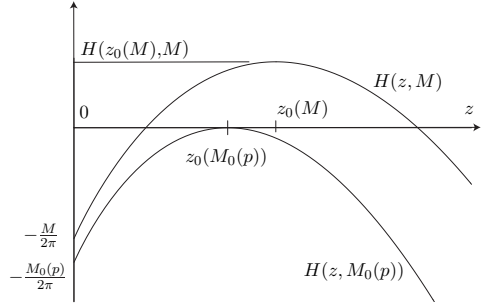}\end{center}\caption{The \emph{method of the trap\/} amounts to prove that $H(z,M)\le 0$ implies that $z=\psi(t)$ is bounded by $z_0(M)$ as long as $H(z_0(M),M)>0$, {\emph i.e.\/} for $M<M_0(p)$. For some $p>4$, the plots of the functions $z\mapsto H(z,M)$ with $M<M(p)$ and $z\mapsto H(z,M_0(p))$ are shown above.}\label{fig:1}\end{figure}

\noindent Recall that the exponents $\sigma\,$, $\rho\,$, $p\,$, $q\,$ and $r$ are related by 
\[\left\{\begin{array}{ll}
\frac 1{\sigma}+\frac 1{\rho}=1\;,\quad& 1<\sigma <2\;,\vspace{.2cm}\\
\frac 1{p} + \frac 1{q} = \frac 1{\rho}\;,& p\,,\;q >2\;,\vspace{.2cm}\\
\frac 1{r} - \frac 1{q} = \frac 12\;,&r>1\;.
\end{array}\right.\]
For the choice $r=4/3\,$, $q=4\,$, it is known, see \cite{MR717827}, that the optimal constant in the HLS inequality is $C_{\rm HLS}=2\sqrt\pi\,$. As a consequence, we have $C_0=\tfrac{4\kappa_\sigma}{\sqrt\pi}\, M^{\frac 1{p} +\frac 14}\,\frac{\sigma}{2-\sigma}$, with $\sigma=\tfrac{4p}{3p-4}\,$. The exponent $p>4$ still has to be chosen. A tedious but elementary computation shows that  there exists $M_0(p)$ such that $H(z_0(M),M)>0$ if and only if $M<M_0(p)$ and $\sup_{p\in(4,+\infty)}M_0(p)=\lim_{p\to+\infty}M_0(p)\approx 0.822663\,$.
\end{proof}

A simple interpolation argument then gives the following corollary.
\begin{corollary}
\label{S2T2}
For any mass $M<M_1$ and all $p\in [1, \infty]\,$, there exists a positive constant $C=C(p, M)$ with $\lim_{M\to 0_+}C(p, M)=0\,$, such that, if $u$ is a solution of \eqref{Eqn:Keller-Segel} as in Lemma~\ref{Lem:Uniform-u},
\[
\nrmrd{u(t)}p \le C\, t^{-(1-\frac{1}{p})}\quad\forall\;t>0\;.
\]
\end{corollary}
\begin{remark}\rm
\label{estimationlp}
Similar decay rates for the $L^p$ norms of the solutions to global Keller-Segel systems have been obtained in a large number of previous references, but always  in slightly different situations. For instance, in \cite{MR2035446}, the authors consider a parabolic-parabolic Keller-Segel system with small and regular initial data. More recently, in \cite{MR2404241} a parabolic-parabolic Keller-Segel system is considered for small initial data and spatial dimension $d\ge 3\,$. On the other hand, a parabolic-elliptic system is treated in \cite{MR2333473} where the equation for the chemo-attractant is slightly different from ours.
\end{remark}

\begin{remark}\rm
The rates obtained in Corollary~\ref {S2T2} are optimal as can easily be checked using the self-similar solutions $(n_\infty,\,c_\infty)$ of \eqref{Eqn:Keller-Segel-Rescaled} defined in Section~\ref{Sec:Intro}. This is the subject of the next section.
\end{remark}

\section{$L^p$ and $H^1$ estimates in the self-similar variables}\label{Sec:Estimates}

Consider now the solution $(n, c)$ defined in the introduction by
(\ref{Eqn:ChangeVar}) and solving (\ref{Eqn:Keller-Segel-Rescaled}). By Corollary \ref{S2T2} we immediately deduce that, for any $p \in (1,\infty]\,$,
\begin{equation}
\label{S3E1}
\nrmrd{n(t)}p \le C_1\quad \forall\;t>0
\end{equation}
for some positive constant $C_1\,$. A direct estimate gives
\[
2\pi\,\left\|\nabla c(t) \right\|_{L^\infty}\le\sup_{x\in\R^2}\int_{\R^2}\frac{n(t,y)}{|x-y|}\;dy\le\underbrace{\sup_{x\in\R^2}\int_{|x-y|\ge 1}\frac{n(t,y)}{|x-y|}\;dy}_{\le M}+\underbrace{\sup_{x\in\R^2}\int_{|x-y|\le 1}\frac{n(t,y)}{|x-y|}\;dy}_{\le\left(2\pi\,\frac{p-1}{p-2}\right)^\frac p{p-1}\, \nrmrd np}
\]
where the last term has been evaluated by H\"older's inequality with $p>2\,$. Hence we obtain
\begin{equation}
\label{S3E2}
\nrmrd{\nabla c(t)}\infty \le C_2\quad \forall\;t>0\;.
\end{equation}
\begin{lemma}
\label{Lem:smallness}
In \eqref{S3E1} and \eqref{S3E2}, the constants $C_1$ and $C_2$ depend on $M$ and are such that
\[
\lim_{M\to 0_+}C_i(M)=0\quad i=1\,,\;2\;.
\]
\end{lemma}
\begin{proof}This result can easily be retraced in the above computations. Details are left to the reader.\\ \end{proof}
With $K=K(x)=e^{|x|^2/2}$, let us rewrite the equation for $n$ as
\begin{equation}
\label{S3E3}
\frac{\partial n}{\partial t}-\frac{1}{K}\,\nabla \cdot \left(K\, \nabla n\right)=-\nabla c\cdot \nabla n +2n+n^2\;.
\end{equation}
We are now interested in the bounds satisfied by the function $n(t)$ in the weighted spaces $L^2(K)$ and $H^1(K)\,$.
\begin{proposition}
\label{S3T1}
For all masses $M\in(0,M_1)$, there exists a positive constant $C$ such that, if $n$ is a solution of \eqref{S3E3} with initial data $n_0\in L^2(K)$ satisfying \eqref{eq:2}, then
\[
\left\| n(t) \right\|_{L^2(K)} \le C\quad \forall\;t>0\;.
\]
\end{proposition}
\begin{proof}
We multiply the equation (\ref{S3E3}) by $n\, K$ and integrate by parts to obtain
\begin{equation}
\label{S3E4}
\frac{1}{2}\,\frac d{dt}\int_{\R^2}|n|^2\;K\,dx+\int_{\R^2} |\nabla n|^2\;K\,dx=-\int_{\R^2} n \nabla c\cdot \nabla n\;K\,dx+2 \int_{\R^2} n^2\;K\,dx+\int_{\R^2} n^3\;K\,dx\;.
\end{equation}
As in \cite[Corollary 1.11]{MR913672}, we recall that for any $q>2$ and $\varepsilon>0\,$, there exists a positive constant $C(\varepsilon, q)$ such that
\[
\int_{\R^2} n^2\;K\,dx\le \varepsilon \int_{\R^2}|\nabla n|^2\;K\,dx+C(\varepsilon,q)\,\|n\|_{L^q(\R^2)}^2\,.
\]
This estimate, (\ref{S3E1}) and (\ref{S3E2}) give a bound of the right hand side of \eqref{S3E4}, namely
\[
\left|\;-\int_{\R^2}n\nabla c\cdot\nabla n\;K\,dx+2 \int_{\R^2} n^2\;K\,dx+\int_{\R^2} n^3\;K\,dx\;\right|\le\varepsilon\int_{\R^2}|\nabla n|^2\;K\,dx+C
\]
up to the multiplication of $\varepsilon$ by a constant that we omit for simplicity, from which we deduce that,
\[
\frac{1}{2}\frac{d}{dt}\int_{\R^2}|n|^2\;K\,dx+(1-\varepsilon)\int_{\R^2} |\nabla n|^2\;K\,dx \le C\;.
\]
We finally use the classical inequality, which is easily recovered by expanding the square in $\int_{\R^2}|\nabla (n\,K)|^2\,K^{-1}\,dx\ge 0\,$, namely
\[
\int_{\R^2} |n|^2\;K\,dx\le \frac 12\int_{\R^2} |\nabla n|^2\;K\,dx
\]
as in \cite{MR913672} to obtain a uniform bound of $n(t)$ in $L^2(K)\,$.\end{proof}

Next we deduce a uniform bound in $H^1(K)\,$.
\begin{corollary} 
\label{S3T2}
Under the assumptions of Proposition~\ref{S3T1}, there exists $T>0$ and $C>0$ such that
\[
\left\| n(t) \right\|_{H^1(K)} \le C\,\max\left\{1,\tfrac{\sqrt T}{\sqrt t}\right\}\quad \forall\;t>0\;.
\]
\end{corollary}
\begin{proof} Since $n$ is a classical solution of (\ref{S3E3}), it also solves the corresponding integral equation,
\[
n(t, x)=S(t)\,n_0(x)-\int_0^tS(t-s)\,(\nabla c\cdot \nabla n)(s)\;ds + \int_0^tS(t-s)\,(2n+n^2)(s)\;ds
\]
where $S(t)$ is the linear semi-group generated by the operator $-K^{-1}\,\nabla \cdot \left(K\,\, \nabla \cdot \right)$ on the space $L^2(K)\,$. Then
\[
\|n(t)\|_{H^1(K)}\le \|S(t)\,n
_0\|_{H^1(K)}+\int_0^t\|S(t-s)\,(\nabla c\cdot \nabla n)(s)\|_{H^1(K)}\;ds\\
+ \int_0^t \|S(t-s)\,(2n+n^2)(s)\|_{H^1(K)}\;ds
\]
Using $\|S(t)\,h\|_{H^1(K)}\le\kappa\,(1+t^{-1/2})\,\|h\|_{L^2(K)}$ for some $\kappa>0\,$, and \eqref{S3E2}, we obtain
\begin{eqnarray*}
\hspace*{-18pt}&&\frac1\kappa\left(\|n(t)\|_{H^1(K)}-\|S(t)\,n_0\|_{H^1(K)}\right)\\
\hspace*{-9pt}&&\le\int_0^t\left(1+\tfrac{1}{\sqrt {t-s}} \right) \|(\nabla c\cdot \nabla n)(s)\|_{L^2(K)}\;ds+ \int_0^t \left(1+\tfrac{1}{\sqrt {t-s}} \right) \|(2n+n^2)(s)\|_{L^2(K)}\;ds\\
\hspace*{-9pt}&&\le\int_0^t\left(1+\tfrac{1}{\sqrt {t-s}} \right)
\|\nabla c\|_{L^\infty(\R^2)}\|\nabla n\|_{L^2(K)}\;ds+ \int_0^t \left(1+\tfrac{1}{\sqrt {t-s}} \right) \left(2\,\|n\|_{L^2(K)}+
\|n\|_{L^\infty(\R^2)}\|n\|_{L^2(K)}\right)\;ds\\
\hspace*{-9pt}&&\le C_2\int_0^t\left(1+\tfrac{1}{\sqrt {t-s}} \right)
\|\nabla n(s)\|_{L^2(K)}\;ds+(2+C_1)\int_0^t \left(1+\tfrac{1}{\sqrt {t-s}} \right)
\|n(s)\|_{L^2(K)}\;ds
\end{eqnarray*}
with $C_1$ defined in~\eqref{S3E1} and $C_2$ in~\eqref{S3E2}. Hence, for any $\tau >0$ fixed, we have
\be{Ineq:Gronwall}
\frac1\kappa\,\|n(t+\tau)\|_{H^1(K)}\le\left(1+\tfrac 1{\sqrt t}\right)\,C_1+C_3\int_0^t\left(1+\tfrac{1}{\sqrt {t-s}} \right) \|n(s+\tau)\|_{H^1(K)}\;ds
\ee
with $C_3=\max\{C_2,2+C_1\}\,$. Let 
\[
H(T)=\sup_{t\in(0,T)}\int_0^t\left(1+\tfrac{1}{\sqrt {t-s}} \right) \|n(s+\tau)\|_{H^1(K)}\,ds\;.
\]
If we choose $T>0$ such that $\tfrac 1{2\kappa}=C_3\int_0^T\!\Big(1+\tfrac 1{\sqrt {T-s}} \Big)\,ds=C_3\big(T+2\sqrt {T}\big)\,$, that is, $T=\Big(\sqrt{1+(2\kappa C_3)^{-1}}-1\Big)^2\,$, then an integration of \eqref{Ineq:Gronwall} on $(0,T)$ gives
\begin{multline*}
\frac1\kappa\,H(T)\le C_1\int_0^T\left(1+\tfrac{1}{\sqrt {T-s}}\right)\left(1+\tfrac 1{\sqrt s}\right)\,ds+C_3\int_0^T\left(1+\tfrac{1}{\sqrt {T-s}}\right)\,H(T)\;ds\\
=\left(\pi+4\sqrt T+T\right)\,C_1 +\frac 1{2\kappa}\,H(T)\;,
\end{multline*}
that is
\[
H(T)\le 2\left(\pi+4\sqrt T+T\right) \,\kappa\,C_1\;.
\]
Injecting this estimate into \eqref{Ineq:Gronwall}, we obtain
\[
\frac1\kappa\,\|n(t + \tau)\|_{H^1(K)}\le\left(1+\tfrac 1{\sqrt t}\right)\,C_1+C_3\,H(T)\le\left(1+\tfrac 1{\sqrt t}\right)\,C_1+2\left(\pi+4\sqrt T+T\right)\,\kappa\,C_1\,C_3
\]
for any $t\in(0,T)\,$. This bounds $\|n(T+\tau)\|_{H^1(K)}$ for any $\tau>0\,$, and thus completes the proof with $C$ given by the right hand side of the above inequality at $t=T\,$. \end{proof}

We shall actually prove that $n(t)$ can be bounded not only in $H^1(K)$ but also in $H^1(n_\infty^{-1})\,$. However, in order to prove that, we need a spectral gap estimate, which is the subject of the next section.

\section{A spectral gap estimate}\label{Sec:f}

Introduce $f$ and $g$ defined by
\[
n(x,t)=n_\infty(x)(1+f(x,t))\qquad \mbox{and} \qquad c(x,t)=c_\infty(x)(1+g(x,t))\;.
\]
By~\eqref{Eqn:Keller-Segel-Rescaled}, $(f,g)$ is solution of the non-linear problem
\begin{equation}\label{eq:nonlinpmeonfg}
\left\{
\begin{array}{ll}
\displaystyle \frac{\partial f}{\partial t} - \mathcal L(t,x,f,g) =-\frac{1}{n_\infty}\,\nabla \cdot \left[f\,n_\infty\,\nabla \left(g\,c_\infty\right)\right] \qquad & x\in\R^2\,,\;t>0\;,\vspace{.3cm}\\
\displaystyle -\Delta (c_\infty\,g) =f\,n_\infty\qquad & x\in\R^2\,,\;t>0\;,
\end{array}\right.
\end{equation}
where $\mathcal L$ is the linear operator given by
\[
\mathcal L(t,x,f,g) =\frac{1}{n_\infty}\,\nabla\cdot \left [n_\infty\,\nabla \left(f- g\, c_\infty\right)\right]\;.
\]
The conservation of mass is replaced here by $\int_{\R^2} f\,n_\infty\,dx=0\,$.
\begin{lemma}\label{lemma:PoincGauss} Let $\sigma$ be a positive real number. For any $g\in H^1\cap L^1(\R^2)$ such that $\int_{\R^2}g\,dx=0$, we have
\[
\int_{\R^2} \left(| \nabla g |^2 + \frac{| x |^2}{4\sigma^2}\,|g|^2 \right)\, dx \ge \tfrac 2\sigma \int_{\R^2} |g|^2\;dx\;.
\]
\end{lemma}
\begin{proof}
The Poincar\'e inequality for the Gaussian measure $d\mu_\sigma(x)=e^{-|x|^2/(2\sigma)}\,dx$ is given by
\[
\sigma\int_{\R^2} |\nabla f|^2\;d\mu_\sigma\ge\int_{\R^2}|f|^2\;d\mu_\sigma\quad\forall\;f\in H^1(d\mu_\sigma)\;\mbox{such that }\int_{\R^2}f\;d\mu_\sigma=0\;.
\]
The result holds with $g=f\,e^{-|x|^2/(4\sigma)}\,$. Notice that for $\sigma=1$,  the second eigenvalue of the harmonic oscillator in $\R^2$ is $2\,$, thus establishing the optimality in both of the above inequalities. The case $\sigma\neq 1$ follows from a scaling argument.
\end{proof}

\begin{proposition}
\label{S5T1}
Consider a stationary solution $n_\infty$ of \eqref{Eqn:Keller-Segel-Rescaled}. There exist a constant $M_2\in(0,8\pi)$ and a function $\Lambda=\Lambda(M)$ such that, for any $M\in(0,M_2)\,$, $\Lambda(M)>0$ and 
\[
\int_{\R^2} | \nabla f |^2\, n_\infty\;dx \ge \Lambda(M) \int_{\R^2} |f|^2\, n_\infty\;dx\quad\forall\;f\in H^1(n_\infty\,dx)\;\mbox{such that }\int_{\R^2}f\,n_\infty\;dx=0\;.
\]
Moreover, $\lim_{M\to 0_+}\Lambda(M)=1\,$. 
\end{proposition}
\begin{proof} We define $h=\sqrt{n_\infty}\,f=\sqrt\lambda\,e^{-|x|^2/4+c_\infty/2}\,f$ with $\lambda=M\,\left(\int_{\R^2}e^{-|x|^2/4+c_\infty/2}\,dx\right)^{-1}\,$. By expanding the square, we find that
\[
\lambda\,|\nabla f |^2\,n_\infty = | \nabla h |^2 +\frac{| x |^2}4\, h^2 +\frac 14\, | \nabla c_\infty |^2\, h^2+h\,\nabla h \cdot (x-\nabla c_\infty) - \frac 12\,x\cdot \nabla c_\infty\,h^2\;.
\]
An integration by parts shows that
\[
\int_{\R^2}h\,\nabla h \cdot x\;dx=-\int_{\R^2}h^2\;dx\;.
\]
Another integration by parts and the definition of $c_\infty$ give
\[
\int_{\R^2} h\,\nabla h \cdot \nabla c_\infty\;dx = \frac 12\int_{\R^2} h^2\,(-\Delta c_\infty)\;dx = \frac 12\int_{\R^2} h^2\,n_\infty\;dx \le \frac 12\,\nrmrd{n_\infty}\infty\int_{\R^2} h^2\;dx\;.
\] 
Recall that by \eqref{Eqn:BifurcationDiagram}, $\lim_{M\to 0_+}\nrmrd{n_\infty}\infty=0\,$. On the other hand, we have
\[
\frac 12 \int_{\R^2} x\cdot \nabla c_\infty \,h^2\;dx \le  \frac{\sigma^2-1}{\sigma^2}\int_{\R^2} \frac{| x |^2}4\, h^2\;dx + \frac 14\,\frac{\sigma^2} {\sigma^2-1}\int_{\R^2} | \nabla c_\infty |^2\, h^2\;dx
\]
for any $\sigma>1\,$. Hence it follows from Lemma~\ref{lemma:PoincGauss} that
\[
\lambda\,\int_{\R^2}|\nabla f |^2\,n_\infty\;dx\ge \underbrace{\left(\frac 2\sigma-1-\frac{\sigma^2\nrmrd{\nabla c_\infty}\infty^2}{4(\sigma^2-1)}-\frac 12\,\nrmrd{n_\infty}\infty\right)}_{\le\Lambda(M)}\underbrace{\int_{\R^2}h^2\;dx}_{=\lambda\int_{\R^2}|f|^2\,n_\infty\;dx}\,.
\]
The coefficient $\Lambda(M)$ is positive for any $M<M_2$ with $M_2>0$, small enough, according to~\eqref{Eqn:BifurcationDiagram},~\eqref{S3E2} and Lemma~\ref{Lem:smallness}. Notice that for each given value of $M<M_2$, an optimal value of $\sigma\in(1,2)$ can be found. \end{proof}

We shall now consider the case of an initial data $n_0$ such that $n_0/n_\infty\in L^2(n_\infty)$, which is a slightly more restrictive case than the framework of Section~\ref{Sec:Estimates}. Indeed, there exists a constant $C>0$ such that for any $x \in \R^2$ with $|x|>1$ we have $|c_\infty + M/(2\pi)\log|x|| \le C$, see~\cite[Lemma~4.3]{BDP}, whence $n_\infty\,K=e^{c_\infty}$ behaves like $O(|x|^{-M/(2\pi)})$ as $|x|\to\infty\,$. If $(n,c)$ is a solution of \eqref{Eqn:Keller-Segel-Rescaled}, then
\[
\frac{\partial n}{\partial t}-n_\infty\,\nabla \cdot \left(\frac 1{n_\infty}\, \nabla n\right)=(\nabla c_\infty-\nabla c)\cdot \nabla n +2n+n^2\;.
\]
\begin{corollary}
Under the assumptions of Theorem~\ref{Thm:Main}, if $M<M_2$, then any solution of \eqref{Eqn:Keller-Segel-Rescaled} is bounded in $L^\infty(\R^+, L^2(n_\infty^{-1}\,dx))\cap L^\infty((\tau,\infty),H^1(n_\infty^{-1}\,dx))$ for any $\tau>0\,$.\end{corollary}
\begin{proof} The uniform bound in $L^2(n_\infty^{-1}\,dx)$ follows from \eqref{S3E4}, up to the replacement of $K$ by $1/n_\infty\,$, which is straightforward. As for the bound in $L^\infty((\tau,\infty),H^1(n_\infty^{-1}\,dx))\,$, one can observe that the linear semi-group $S(t)$ generated by the self-adjoint operator $-n_\infty\,\nabla \cdot \left(\tfrac{1}{n_\infty}\, \nabla n\right)$ on the space $L^2(n_\infty^{-1})\,$, with domain $H^2(n_\infty^{-1})$, satisfies $\left\|S(t)\,n_0\right\|_{H^1(n_\infty^{-1}\,dx)}\le\tfrac\kappa{\sqrt t}\,\left\|n_0\right\|_{L^2(n_\infty^{-1}\,dx)}$ for some $\kappa>0\,$, see for instance \cite[Theorem VII.7]{MR697382}. The estimate then follows as in Corollary~\ref{S3T2}.\end{proof}

\section{Proof of Theorem \ref{Thm:Main}}\label{Sec:ExponentialConvergence}
This Section is devoted to the proof of our main result. If we multiply equation \eqref{eq:nonlinpmeonfg} by $f\, n_\infty$ and integrate by parts, we get
\be{Eqn:fninfini}
\frac{1}{2}\frac{d}{dt}\int_{\R^2}|f|^2 \,n_\infty\;dx+
\int_{\R^2} |\nabla f|^2 \,n_\infty\;dx=\int_{\R^2}\nabla f\cdot\nabla\left(g\,c_\infty \right)\,n_\infty\;dx +\int_{\R^2} \nabla f\cdot\,\nabla(g\,c_\infty)\,f\, n_\infty\;dx\;.
\ee
The first term of the right hand side can be estimated as follows. By the Cauchy-Schwarz inequality, we know that
\[
\int_{\R^2} \nabla f\cdot\nabla\left(g\,c_\infty \right)\,n_\infty\;dx\le\nrmn{\nabla f}2\,\nrmn{\nabla (g\, c_\infty)}2\;.
\]
By H\"older's inequality, for any $q>2$ we have
\begin{equation*}
  \nrmn{\nabla (g\, c_\infty)}2 \le M^{1/2-1/q}\,\|n_\infty\|_{L^\infty(\R^2)}^{1/q}\nrmrd{\nabla (g\, c_\infty)}q\;.
\end{equation*}
The HLS inequality with $1/p=1/2+1/q$ then gives
\begin{equation*}
  \nrmrd{\nabla (g\, c_\infty)}q\le \frac 1{2\pi}\left(\int_{\R^2}\left|\,(f\,n_\infty)*\tfrac 1{|\,\cdot\,|}\,\right|^q\,dx\right)^\frac 1q\le\frac{C_{\rm HLS}}{2\pi}\,\nrmrd{f\,n_\infty}p\;.
\end{equation*}
By H\"older's inequality, $\nrmrd{f\,n_\infty}p\le\nrmn f2\,\nrmrd{n_\infty}{q/2}^{1/2}\,$, from which we get
\begin{equation}\label{eq:drtede13}
\int_{\R^2}\nabla f\cdot\,\nabla(g\,c_\infty)\,f\, n_\infty\;dx\le C_*\,\|f\|_{L^2(n_\infty\,dx)}\nrmn{\nabla f}2
\end{equation}
where $C_*=C_*(M):={C_{\rm HLS}}\,(2\pi)^{-1}\,M^{1/2-1/q}\,\nrmrd{n_\infty}{q/2}^{1/2}\,\|n_\infty\|_{L^\infty(\R^2)}^{1/q}$ goes to 0 as $M \to 0\,$.

As for the second term in the right hand side of~\eqref{Eqn:fninfini}, using $g\,c_\infty = c-c_\infty$ and the Cauchy-Schwarz inequality, we have
\begin{eqnarray*}
\int_{\R^2}\nabla f\cdot\,\nabla(g\,c_\infty)\,f\, n_\infty\;dx  &\le& \|\nabla c - \nabla c_\infty\|_{L^\infty(\R^2)}\,\|f\|_{L^2(n_\infty\,dx)}\,\|\nabla f\|_{L^2(n_\infty\,dx)}\\
&\le&  \left( \|\nabla c\|_{L^\infty(\R^2)} + \|\nabla c_\infty\|_{L^\infty(\R^2)} \right)\,\|f\|_{L^2(n_\infty\,dx)}\,\|\nabla f\|_{L^2(n_\infty\,dx)}\;.
\end{eqnarray*}
We observe that $\nabla (g\,c_\infty)=\nabla c-\nabla c_\infty$ is uniformly bounded since $\nrmrd{\nabla c}\infty\le C_2(M)$ by \eqref{S3E2}, and $\nrmrd{\nabla c_\infty}\infty$ is also bounded by $C_2(M)$, for the same reasons.
\begin{equation}\label{eq:gauchde13}
\int_{\R^2} \nabla f\cdot\,\nabla(g\,c_\infty)\,f\, n_\infty\;dx \le 2\,C_2(M)\, \|f\|_{L^2(n_\infty\,dx)}\, \|\nabla f\|_{L^2(n_\infty\,dx)}\;.
\end{equation}
Moreover, according to Lemma~\ref{Lem:smallness}, we know that $\lim_{M\to 0_+}C_2(M)=0\,$.

By Proposition~\ref{S5T1}, $\nrmn f2\le\nrmn{\nabla f}2/\sqrt{\Lambda(M)}$ with $\lim_{M \to 0^+}\Lambda(M)=1\,$. Collecting~\eqref{eq:drtede13} and~\eqref{eq:gauchde13}, we obtain
\begin{equation*}
  \frac{1}{2}\frac{d}{dt}\int_{\R^2}|f|^2 \,n_\infty\;dx \le -\left[1-\gamma(M)\right] \int_{\R^2}|\nabla f|^2 \,n_\infty\;dx\quad\mbox{with}\quad\gamma(M):=\frac{C_*(M)+2\,C_2(M)}{\sqrt{\Lambda(M)}}\;.
\end{equation*}
We observe that $\lim_{M \to 0^+}\gamma(M)=0\,$. As long as $\gamma(M)<1\,$, we can use again Proposition~\ref{S5T1} to get
\be{eq:decayrate}
\frac{1}{2}\frac{d}{dt}\int_{\R^2}|f|^2 \,n_\infty\;dx\le -\,\delta\int_{\R^2}|f|^2 \,n_\infty\;dx\quad\mbox{with}\quad\delta=\Lambda(M)\,\left[1-\gamma(M)\right]\;.
\ee
Using a Gronwall estimate, this establishes the decay rate of $\nrmn f2=\nrmrd{\tfrac{n-n_\infty}{\sqrt{n_\infty}}}2\,$.

If $n_1$ and $n_2$ are two solutions of~\eqref{Eqn:Keller-Segel-Rescaled} in $C^0(\R^+,L^1(\R^2))\cap L^\infty((\tau,\infty)\times\R^2)$ for any $\tau>0\,$, Inequality~\eqref{eq:decayrate} also holds for $f=(n_2-n_1)/n_\infty$. As a consequence, if the initial condition is the same, then $n_1=n_2\,$, which proves the uniqueness result and concludes the proof of Theorem~\ref{Thm:Main}.\\ \qed

\begin{remark}\rm
\label{Rem:galmathcalL}
Proposition~\ref{S5T1} and~\eqref{eq:drtede13} rely on rather crude estimates of the spectral gap of the linear operator $\mathcal L\,$, defined on $L^2(n_{\infty})\,$, with domain $H^2(n_\infty)\,$. The operator has been divided in two parts which are treated separately, one in Proposition~\ref{S5T1}, the other one in~\eqref{eq:drtede13}. It would probably be interesting to study the operator $\mathcal L$ as a whole, trying to obtain an estimate of its spectral gap in $L^2(n_{\infty})$ without any smallness condition.
\end{remark}

\bigskip\noindent\emph{Acknowledgements.\/} M.E. is supported by Grant MTM2008-03541 and the RTRA \emph{Sciences math\'ematiques de Paris.\/} A.B. and J.D. are supported by the ANR projects \emph{IFO} and \emph{EVOL\/}.


\end{document}